\documentclass[12pt]{amsart}
\usepackage{amsmath,amsthm,amssymb,amsfonts,enumerate,color}
\usepackage{amsmath, amsthm, amsfonts, amssymb, color}
\usepackage{mathrsfs}
\usepackage{ amsmath, amsthm, amsfonts, amssymb, color}
 \usepackage{mathrsfs}
\usepackage{amsfonts, amsmath}
 \usepackage{amsmath,amstext,amsthm,amssymb,amsxtra}
 \usepackage{txfonts} 
 \usepackage[colorlinks, citecolor=blue,pagebackref,hypertexnames=false]{hyperref}
 \allowdisplaybreaks
 \usepackage{pgf,tikz}
\begin{document}
\baselineskip 18pt
\hfuzz=6pt

\newtheorem{theorem}{Theorem}[section]
\newtheorem{prop}[theorem]{Proposition}
\newtheorem{lemma}[theorem]{Lemma}
\newtheorem{definition}[theorem]{Definition}
\newtheorem{cor}[theorem]{Corollary}
\newtheorem{example}[theorem]{Example}
\newtheorem{remark}[theorem]{Remark}
\newtheorem{assumption}[theorem]{Assumption}
\newcommand{\ra}{\rightarrow}
\renewcommand{\theequation}
{\thesection.\arabic{equation}}
\newcommand{\ccc}{{\mathcal C}}
\newcommand{\one}{1\hspace{-4.5pt}1}

\def\HSL { H^1_{L,S}(X) }

\def \Gg {\widetilde{{\mathcal g}}_{L}}
\def \GG {{\mathcal g}_{L}}
\def \SL {\sqrt[m]{L}}
\def \sq {\sqrt}
\def \GL {{\mathcal G}_{\lambda,L}^{\ast}}

\def \l {\lambda}
\def \gL{{\widetilde {\mathcal G}}_{\lambda, L}^{\ast}}
\def \RN {\mathbb{R}^n}
\def\RR{\mathbb R}
\def \Real {\mathbb R}
\newcommand{\nf}{\infty}
 \def \Lips  {{   \Lambda}_{L}^{ \alpha,  s }(X)}
\def\BL {{\rm BMO}_{L}(X)}
\def\HAL { \mathbb{F}\dot{\mathbb{H}}_{L,at,M}^{\lambda}({\mathbb R}^n)}
\def\HML { \mathbb{F}\dot{\mathbb{H}}_{L,mol,M}^{\lambda}({\mathbb R}^n) }
\def\HM{ H^p_{L, {mol}, 1}(X) }
\def\CL {\mathscr{L}_L^{p,\lambda}(\RN)}
\def \M {L^{p,\lambda}}

\def\HSL { H^p_{L, S_h}(X) }
\newcommand\mcS{\mathcal{S}}
\newcommand\mcB{\mathcal{B}}
\newcommand\D{\mathcal{D}}
\newcommand\C{\mathbb{C}}
\newcommand\N{\mathbb{N}}
\newcommand\R{\mathbb{R}}
\newcommand\G{\mathbb{G}}
\newcommand\T{\mathbb{T}}
\newcommand\Z{\mathbb{Z}}
\allowdisplaybreaks

 \title[Poisson integrals of Schr\"odinger operators with Morrey traces ]
 {On characterization of Poisson integrals of Schr\"odinger operators with Morrey traces}

\author[Liang Song, Xiaoxiao Tian and Lixin Yan]{Liang Song, Xiaoxiao Tian and Lixin Yan}
\thanks{{\it {\rm 2010} Mathematics Subject Classification:}
  42B37,  42B35, 47B38.}
\thanks{{\it Key words and phrases:} Schr\"odinger  operators, Dirichlet problem, Morrey spaces, Campanato spaces,  Poisson semigroup.
 }

\address{Liang Song, Department of Mathematics, Sun Yat-sen
University, Guangzhou, 510275, P.R. China} \email{songl@mail.sysu.edu.cn}
\address{Xiaoxiao Tian, Department of Mathematics, Sun Yat-sen
University, Guangzhou, 510275, P.R. China} \email{tianxx3@mail2.sysu.edu.cn}
\address{Lixin Yan, Department of Mathematics, Sun Yat-sen
University, Guangzhou, 510275, P.R. China} \email{mcsylx@mail.sysu.edu.cn}

\begin{abstract}

 Let $L$ be a Schr\"odinger operator of the form $L=-\Delta+V$ acting on $L^2(\mathbb R^n)$ where the nonnegative
potential $V$ belongs to  the reverse H\"older class    $B_q$ for some $q\geq n.$
In this article we will
show that
 a function  $f\in L^{2, \lambda}(\RN), 0<\lambda<n$ is the trace of the solution of
${\mathbb L}u=-u_{tt}+L u=0,  u(x,0)= f(x),$
 where $u$ satisfies  a Carleson type condition
\begin{eqnarray*}
 \sup_{x_B, r_B} r_B^{-\lambda}\int_0^{r_B}\int_{B(x_B, r_B)}  t|\nabla   u(x,t)|^2 {dx dt }  \leq C <\infty.
\end{eqnarray*}
Its  proof heavily relies on
 investigate the intrinsic relationship between the classical  Morrey spaces   and
the new Campanato  spaces
 $\mathscr{L}_L^{2,\lambda}(\RN)$ associated to the operator $L$, i.e.
$$\mathscr{L}_L^{2,\lambda}(\mathbb{R}^n)= {L}^{2,\lambda}(\mathbb{R}^n).
$$
Conversely, this Carleson  type  condition characterizes  all the ${\mathbb L}$-harmonic functions whose traces belong to
the space $L^{2, \lambda}(\RN)$ for all $ 0<\lambda<n$. This extends the previous results of \cite{FJN, DYZ, JXY}.
\end{abstract}

\maketitle

\section{Introduction}

\setcounter{equation}{0}

 Consider Schr\"odinger operators
\begin{align}\label{e4.1}
L=-\Delta+V(x) \qquad  {\rm on} \  \RN, \ \quad  n\geq 3,
\end{align}
 where $V\in L^1_{\rm loc}(\RN)$ is   nonnegative  and not identically zero. It follows that
 the operator $L$ is a self-adjoint positive definite operator.  From the Feynman-Kac formula,
  the kernel $p_t(x,y)$ of the semigroup
$e^{-tL}$ satisfies the estimate
\begin{align}\label{e4.1}
0\leq p_t(x,y)\leq \frac{1}{(4\pi t)^{n/2}}e^{-\frac{|x-y|^2}{4t}}.
\end{align}
In addition, suppose that  $V(x)$  also belong to the reverse
H\"older class $B_q$ for some $q\geq n$,
which by definition means that   $V\in L^{q}_{\rm loc}(\RN), V\geq 0$, and
there exists a  constant $C>0$ such that   the reverse H\"older inequality
$$
\left ( \frac{1}{|B|}\int_B {V(x)}^q \, dx\right )^{1/q}\leq C  \frac{1}{|B|}\int_B V(x) \, dx
$$
holds for each ball  $B$ in $\RN$. We refer to \cite{DYZ, MSTZ, Shen} for the properties of these Schr\"odinger operators.

For $f\in  L^p(\RN)$, $1\leq p<  \infty,$
it is well known that the Poisson extension $u(x,t)=e^{-t\sqrt{L}}f(x), t>0, x\in\RN$, is a solution to the equation
\begin{eqnarray}\label{e4.7}
{\mathbb L}u=-u_{tt}+{L} u =0\ \ \ {\rm in }\ {\mathbb R}^{n+1}_+
\end{eqnarray}
with the boundary data $f$ on $\RN$. The equation  ${\mathbb L}u =0$ is understood in the weak sense, that is,
  $u\in {W}^{1, 2}_{{\rm loc}} ( {\mathbb R}^{n+1}_+) $ is a weak solution of ${\mathbb L}u =0$   if it satisfies
$$\int_{{\mathbb R}^{n+1}_+} {\nabla}u\cdot {\nabla}\psi \,dY+
 \int_{{\mathbb R}^{n+1}_+} V u\psi \,dY=0,\ \ \ \ \forall \psi\in C_0^{1}({\mathbb R}^{n+1}_+).
 $$
At the end-point space $L^{ \infty}(\RN)$, the study of singular integrals has a natural substitution, the BMO space, i.e. the space
of functions of bounded mean oscillation.
It was shown in \cite{DYZ}   that
 a
 $  {\rm BMO}_{{L}}(\RN)$ function   is the trace of the solution of
 $-u_{tt}+ L u=0,  u(x,0)= f(x),$
 whenever $u$ satisfies
\begin{eqnarray} \label{e4.8}
 \sup_{x_B, r_B} r_B^{-n}\int_0^{r_B}\int_{B(x_B, r_B)}  t|\nabla  u(x,t)|^2 {dx dt }  \leq C<\infty,
\end{eqnarray}
where $\nabla=(\nabla_x, \partial_t).$
Conversely, this Carleson   condition characterizes  all the ${\mathbb L}$-harmonic functions whose traces belong to
the space ${\rm BMO}_{{L}}(\RN)$ associated to an operator $L$,
which  extends the  analogous characterization founded by Fabes, Johnson and Neri in \cite{FJN}
for   the classical BMO space  of John-Nirenberg. See also Chapter 2 of the standard textbook \cite{SW}.

The main goal of this paper is to continue this line of research (\cite{DYZ, FJN, SW}) to study
  the cases of $f$ in Morrey spaces. Recall that  Morrey spaces were introduced in 1938 by C.B.  Morrey \cite{Mo}
in relation to regularity problems of solutions to partial differential equations.
Recall that for every $1\leq p<\infty$ and $\lambda\geq 0$,
the Morrey
space ${L}^{p,\lambda}(\RN)$  is defined as
 \begin{align}\label{e1.1}
{L}^{p,\lambda}(\RN)=\left\{ f\in L^p_{\rm loc}(\RN): \ \sup\limits_{x\in \RN;\ r>0} r^{-\lambda}
\int_{B(x, r)} |f(y)|^p \ dy<\infty\right\}.
\end{align}
This is a Banach space with respect to the norm
 \begin{align}\label{e1.2}
\|f\|_{L^{p,\lambda}(\RN)}=\left(\sup\limits_{x\in \RN;\ r>0} r^{-\lambda}
\int_{{  B}(x, r)} |f(y)|^p \ dy\right)^{1/p}<\infty.
\end{align}
It is known that $L^{p, 0}(\RN)= L^p(\RN)$ and $L^{p, n}(\RN)= L^{\infty}(\RN)$. When $\lambda>n$, the space ${L}^{p,\lambda}(\RN)$ is trivial ($L^{p, \lambda}(\RN)=\{0\}$).
 In the case $\lambda\in (0, n]$,
the space ${L}^{p,\lambda}(\RN)$ is non-separable. We refer to \cite{JTW,P,AX2,Ca,Na,RSS,Ta} for more properties of the Morrey spaces.

The following theorem is our main result of this paper.
 \begin{theorem}\label{th4.2}
 Suppose $V\in B_q$ for some $q\geq n$ and $0<\lambda<n$.
We denote by  ${\rm HL}^{2,\lambda}_L(\Real_+^{n+1})$  the class of all $C^1$-functions $u(x,t)$
of the solution of ${\mathbb L}u=0$  in $\Real_+^{n+1} $ such that
\begin{eqnarray}\label{e4.8}
\|u\|^2_{{\rm HL}^{2,\lambda}_L(\Real_+^{n+1})}&=& \sup_{x_B, r_B}     r_B^{-\lambda}
\int_0^{r_B}\!\int_{B(x_B, r_B)} t | \nabla u(x,t) |^2
{dx dt }  <\infty,
\end{eqnarray}
\noindent where  $\nabla=(\nabla_x, \partial_t)$.
 Then we have
 \begin{itemize}
\item[(1)]  If $f\in L^{2,\lambda}(\RN)$, then   the function $u=e^{-t\sqrt{L}}f$
is in ${\rm HL}^{2,\lambda}_L(\Real_+^{n+1})$,   and
 $$
 \|u\|_{{\rm HL}^{2,\lambda}_L(\Real_+^{n+1})}\leq C \|f\|_{L^{2,\lambda}(\RN)}.
 $$

\item[(2)]  If $u\in {\rm HL}^{2,\lambda}_L(\Real_+^{n+1})$, then    there exists some
$f\in L^{2,\lambda}(\Real^{n})$ such that $u(x,t)=e^{-t \sqrt{L}}f(x)$,
and
$$
\|f\|_{L^{2,\lambda}(\RN)}\leq C\|u\|_{{\rm HL}^{2,\lambda}_L(\Real_+^{n+1})}.
$$
 \end{itemize}
\end{theorem}

We would like to mention that  in the case of $L=-\Delta$,   Theorem \ref{th4.2}
was obtained by Jiang, Xiao and Yang \cite{JXY}. When $L$ is  the Schr\"odinger operators as in (\ref{e4.1}) above,
the proof of (1) of  our Theorem \ref{th4.2}  follows by the standard argument to use  the definition of Morrey spaces and the full gradient estimates
on the kernel of the Poisson semigroup in the $(x,t)$ variables under the assumption on $V\in B_q$ for some $q\geq n$.
To prove (2) of  our Theorem \ref{th4.2}, we need to investigate the intrinsic relationship between the  Morrey space   and
the Campanato  spaces
 $\mathscr{L}_L^{p,\lambda}(\RN)$ associated to operators, which was introduced and studied in
 \cite{DDSY,  DXY, DY1,  DY2}.
Following \cite{DXY}, we say that a function $f$ (with appropriate bound on its size $|f|$) belongs to the space
  (where $1\leq p<\infty$ and $\lambda>0$), provided
\begin{eqnarray}\label{e1.6}
 \|f\|_{{\mathscr L}^{p, \lambda}_L(\RN)}=
 \left(\sup\limits_{x\in\RN, \ r>0} r^{-\lambda}\int_{ B (x,r)}
|f(y)-e^{-r^2L}f(y)|^p \ dy\right)^{1/p}
 <\infty.
\end{eqnarray}
 We  will prove  that  for every  $1\leq p<\infty$ and $0<\lambda<n$,
the Campanato spaces  ${\mathscr L}^{p,\lambda}_L(\RN)$
(modulo the kernel  spaces)   coincides  with  the Morrey spaces  $\M(\mathbb{R}^n)$, i.e.,
\begin{align}\label{ee1.1}
\mathscr{L}_L^{p,\lambda}(\mathbb{R}^n)/\mathcal{K}_{L,p}= {L}^{p,\lambda}(\mathbb{R}^n)
\quad {\rm and} \quad \|f\|_{\CL}\approx\|f-\lim\limits_{t\to +\infty} e^{-tL}f\|_{\M}.
\end{align}
See Theorem~\ref{th3.1} below.
Observe that  for the Schr\"odinger operator $L$ as in (\ref{e4.1})   under the additional assumption $V\in B_q$ for some $q\geq n,$
we have that
for $0<\lambda<n, 1\leq p<\infty$   there holds: $\mathcal{K}_{L,p}=\mathcal{K}_{\sqrt{L},p}=\{0\}$, and hence
\begin{align}\label{e1.7}
\mathscr{L}_L^{p,\lambda}(\mathbb{R}^n)= \mathscr{L}_{\sqrt{L}}^{p,\lambda}(\mathbb{R}^n)
= {L}^{p,\lambda}(\mathbb{R}^n).
\end{align}
From the equivalence \eqref{e1.7} between  the  Morrey spaces   and
the Campanato  spaces
 $\mathscr{L}_L^{p,\lambda}(\RN)$,
  we follow the line of the proof as in  \cite{FJN, DYZ,JXY}
  to obtain the proof of  (2) of Theorem \ref{th4.2}.

Throughout, the letters $C$ and $c$ will  denote (possibly different)
constants that are independent of the essential variables.

\medskip

\medskip
\section{Properties of Campanato spaces associated to  operators}
\setcounter{equation}{0}

In this section, we assume
  that $L$ is a linear operator on $L^2({\mathbb  R}^n)$ which generates
an analytic semigroup $e^{-tL}$ with a  kernel $p_t(x,y)$  satisfying
\begin{equation}
|p_t(x,y)|\leq Ct^{-{n/m}} \left( 1 +{\frac{|x-y|}{t^{1/m}}}\right)^{-(n+\epsilon)},
 \quad\forall x,y\in {\mathbb R}^n,
\label{e2.1}
\end{equation}
where $C, m$ and $\epsilon$ are positive constants.

 We now define the class
of functions that the operators $e^{-tL}$ act upon. Fix $1\leq p<\infty.$
For any  $\beta>0$, a complex-valued function $f\in L^p_{\rm loc}({\mathbb R}^n)$ is
said to be a  function of type
$(p;  \beta)$   if $f$ satisfies
\begin{equation}
\left(\int_{{\mathbb R}^n}{\frac{|f(x)|^p}{(1+|x|)^{n+\beta}}}dx\right)^{1/p}\leq C<\infty.
\label{e2.3}
\end{equation}
We denote by ${\mathcal M}_{(p;  \beta)}$ the collection of all  functions
of type $(p; \beta)$. If $f\in {\mathcal M}_{(p;  \beta)},$  the norm
of $f\in {\mathcal M}_{(p;  \beta)}$ is defined by
$$
\|f\|_{{\mathcal M}_{(p;   \beta)}}
=\inf\{C> 0:\ \ (\ref{e2.3})  \ {\rm holds}\}.
$$
It is not hard to see that
  ${\mathcal M}_{(p;  \beta)}$ is a complex Banach space
under $
\|f\|_{{\mathcal M}_{(p;   \beta)}}<\infty$.
For any given  operator $L$, let
\begin{eqnarray}
\label{e2.4}
  {\Theta }(L)=\sup\big\{\epsilon>0:\ \ (\ref{e2.1})\ {\rm holds} \ \big\}
\end{eqnarray}
and write
 \begin{eqnarray} \label{e2.5}
{\mathcal M}_p=
\left \{
\begin{array}{ll}
{\mathcal M}_{(p; {\Theta }(L))} &\ \ \ \ \ \  \ \ \ {\rm if } \ {\Theta }(L) <\infty;\\\\
  \bigcup\limits_{\beta:\
0<\beta<\infty}
{\mathcal M}_{(p;  \beta)} &\ \ \ \ \ \ \ \ \ {\rm if} \ {\Theta }(L)=\infty.
\end{array}
\right.
\end{eqnarray}
Note that if $L=-\Delta$ or $L=\sqrt{-\Delta}$ on ${\mathbb R}^n$,
then ${\Theta }(-\Delta)=+\infty$ or ${\Theta }(\sqrt{-\Delta})=1,$ respectively.
For any $(x,t)\in {\mathbb R}^{n+1}_+$ and $f\in {\mathcal M}_p$, define
\begin{eqnarray}
 e^{-tL}f(x)=\int_{{\mathbb R}^n} p_t(x,y)f(y)dy.
\label{e2.6}
\end{eqnarray}
 It follows from
(\ref{e2.1}) that
  the operators $P_tf$  is well defined. Following \cite{DXY}, we can define
  Campanato
  spaces    $\mathscr{L}_L^{p,\lambda}(\mathbb{R}^n)$ associated to an operator $L$ as follows.

 \begin{definition}\label{def2.1}
 Let    $1\leq p<\infty$ and $0<\lambda< n$. We say that
  a function $f\in \mathcal{M}_p$    belongs to the space
   $\mathscr{L}_L^{p,\lambda}(\mathbb{R}^n)$ associated to an operator $L$,
   if
\begin{align}\label{e2.7}
\|f\|_{\mathscr{L}_L^{p,\lambda}(\RN)}=\Big(\sup\limits_{B\subset \mathbb{R}^n}
r_B^{-\lambda}\int_B |f(x)-e^{-r^m_BL}f(x)|^p \ dx\Big)^{1/p}<\infty,
\end{align}
where $m$ is a fixed positive constant in \eqref{e2.1}.
\end{definition}

It can be verified that
  $\mathscr{L}_L^{p,\lambda}(\mathbb{R}^n)/ \mathcal{K}_{L,p}$ is a Banach space,
  where   $\mathcal{K}_{L,p}$ is  the kernel space and defined  by
\begin{align}\label{e2.8}
\mathcal{K}_{L,p}=\{f\in \mathcal{M}_p:\ e^{-tL}f(x)=f(x) \  {\rm for \ almost \ all}
\  x\in\mathbb{R}^n \ {\rm and \ all} \  t>0\}.
\end{align}
It is shown  in \cite{DXY} that
the space
  $\mathscr{L}_{-\Delta}^{p,\lambda}(\mathbb{R}^n)$ coincides with
$\mathscr{L}^{p,\lambda}(\mathbb{R}^n)$ for $0<\lambda< n$,and a
 necessary and sufficient condition for the classical space
 ${\mathscr L}^{p, \lambda}(\RN) \subseteq {\mathscr  L}_L^{p, \lambda}(\RN)$   with
 $
\|f\|_{{\mathscr  L}_L^{p, \lambda}(\RN)}\leq C\|f\|_{{\mathscr L}^{p, \lambda}(\RN)}
$
is that  for
every
$\
\! t>0$,
$e^{-tL}(1)=1$ almost everywhere.

Recall that the classical Campanato spaces were introduced in 1963
by S. Campanato \cite{Ca}, which are a generalization of the BMO spaces of functions of
bounded mean oscillation introduced by F. John and L. Nirenberg \cite{JN}.
Recall that for every   $1\leq p<\infty$ and $\lambda>0$,
the Campanato  space $\mathscr{L}^{p,\lambda}(\RN)$ is defined as
 \begin{align}\label{e1.4}
\mathscr{L}^{p,\lambda}(\RN)=\left\{ f\in L^p_{\rm loc}(\RN): \ \|f\|_{\mathscr{L}^{p,\lambda}(\RN)}
<\infty  \right\}
\end{align}
with the Campanato  seminorm being given by
\begin{align}\label{e1.5}
\|f\|_{\mathscr{L}^{p,\lambda}(\RN)}=\left(\sup\limits_{x\in\RN, \ r>0} r^{-\lambda} \int_{B(x,r)}
|f(y)-f_{B(x,r)}|^p \ dy\right)^{1/p}<\infty.
\end{align}
 The relationship between Morrey spaces
 and Campanato spaces is the following important result (see \cite{JTW,P}):
 \begin{itemize}
\item[(i)]  $\mathscr{L}^{p,\lambda}(\RN)/{\mathbb C}  = L^{p,\lambda}(\RN)$,
 when $\lambda\in [0, n)$ and ${\mathbb C}$ denotes the space of all constant functions;

\item[(ii)]  $\mathscr{L}^{p,n}(\RN)= {\rm BMO}(\RN)$, when $\lambda=n;$

\item[(iii)]  $\mathscr{L}^{p,\lambda}(\RN)/{\mathbb C} = {\rm Lip}_{\alpha}({\mathbb R}^n)$ with $\alpha=(\lambda-n)/p$,
 when $\lambda\in (n, n+p)  $  and   ${\rm Lip}_{\alpha}({\mathbb R}^n)$ denotes the homogeneous Lipschitz space  in $ \RN$.
 \end{itemize}

\subsection{The intrinsic relationship between Campanato spaces associated to  operators and classcial Morrey spaces}

 The  goal  of this subsection is to prove the following theorem.

\begin{theorem}\label{th3.1}
  For every  $1\leq p<\infty$ and $0<\lambda<n$,
Campanato   spaces ${\mathscr L}^{p,\lambda}_L(\RN)$
(modulo the kernel  spaces)   coincide  with the classical Morrey spaces $\M(\mathbb{R}^n)$, i.e.,
$$
\mathscr{L}_L^{p,\lambda}(\mathbb{R}^n)/\mathcal{K}_{L,p}={L}^{p,\lambda}(\mathbb{R}^n)
\quad {\rm and} \quad \|f\|_{\CL}\approx\|f-\lim\limits_{t\to +\infty} e^{-tL}f\|_{\M(\RN)},
$$
where $\mathcal{K}_{L,p}$ is given in \eqref{e2.8}.
More precisely, for every $f\in \CL$, the mapping $f\mapsto f-\lim\limits_{t\to +\infty}
e^{-tL}f$ is bijective and bicontinuous from $\mathscr{L}_L^{p,\lambda}(\mathbb{R}^n)/\mathcal{K}_{L,p}$
to ${L}^{p,\lambda}(\mathbb{R}^n)$.
\end{theorem}

The proof of Theorem~\ref{th3.1} is based the following lemmas.

\begin{lemma}\label{le2.3}
 Let $1\leq p<\infty$ and $0<\lambda<n$. Suppose  $f\in \M(\mathbb{R}^n)$. Then
 for every $t>0 $ and $x\in \RN,$  $|e^{-tL}f(x)|\leq Ct^{\lambda-n\over pm}\|f\|_{\M(\RN)}.$
As a consequence, we have
\begin{align*}
f\in \CL  \quad  \quad  {\rm and} \quad \quad \|f\|_{\CL}\leq C\|f\|_{\M(\RN)}.
\end{align*}
\end{lemma}

\begin{proof}  Let  $f\in \M(\mathbb{R}^n)$. From  kernel estimate \eqref{e2.1} of the semigroup $e^{-tL}$, we have
\begin{eqnarray*}
|e^{-tL}f(x)|&\leq& Ct^{-n/m} \int_{\RN} \left( 1 +{\frac{|x-y|}{t^{1/m}}}\right)^{-(n+\epsilon)} |f(y)|dy\nonumber\\
&\leq& Ct^{-n/m}\int_{B(x, 2t^{1/m})}   |f(y)|dy
 +Ct^{-n/m}\sum_{k=2}^{\infty}\int_{B(x, 2^{k}t^{1/m})\backslash
 B(x, 2^{k-1}t^{1/m})} \left( 1 +{\frac{|x-y|}{t^{1/m}}}\right)^{-(n+\epsilon)} |f(y)|dy\nonumber\\
&\leq&
  C\sum_{k=1}^{\infty} 2^{-k(n+\epsilon)}t^{-n/m} \int_{ B(x, 2^{k}t^{1/m})}
 |f(y)|dy.
 \end{eqnarray*}
By  H\"older's inequality, we  obtain that for every $k\in{\mathbb N},$
 \begin{eqnarray*}
 t^{-n/m}\int_{ B(x, 2^{k}t^{1/m})}  |f(y)|dy
&\leq&Ct^{-n/m}(2^kt^{1/m})^{(1-{1\over p})n+ {\lambda\over p}}
\left({1\over (2^kt^{1/m})^{\lambda}} \int_{B(x, 2^{k}t^{1/m})}   |f(y)|^pdy\right)^{1/p}\nonumber\\
&\leq&
C2^{(n+{\lambda-n\over p})k}t^{\lambda-n\over pm}\|f\|_{\M(\RN)},
\end{eqnarray*}
which gives
\begin{eqnarray}\label{e2.10}
|e^{-tL}f(x)| \leq
  C\sum_{k=1}^{\infty} 2^{-k(n+\epsilon)} 2^{(n+{\lambda-n\over p})k}
  t^{\lambda-n\over pm}\|f\|_{\M(\RN)}
 \leq    C t^{\lambda-n\over pm}\|f\|_{\M(\RN)}.
 \end{eqnarray}

It follows from  $f\in \M(\mathbb{R}^n)$  that  $
\int_{\mathbb{R}^n}  {|f(x)|^p}{(1+|x|)^{-(n+{\beta})}}\, dx\leq C\|f\|_{\M(\RN)}^p
 $
  for any ${\beta}>0$, and so
   $f\in  \mathcal{M}_p$.
Note that for any  ball $B=B(x_B, r_B)\subset \mathbb{R}^n$,  we apply estimate (\ref{e2.10}) to obtain
\begin{eqnarray*}
\left(\frac{1}{r_B^\lambda}\int_B |f-e^{-r_B^m L}f(x)|^p \, dx \right)^{1/p}
&\leq&  \big\|f\big\|_{\M(\RN)} +
\left(\frac{1}{r_B^\lambda}\int_B | e^{-r_B^m L}f(x)|^p \, dx \right)^{1/p}\\
&\leq&  C\big\|f\big\|_{\M(\RN)}.
\end{eqnarray*}
Taking the supremum over all the balls $B$ in $\RN$, we have that  $\|f\|_{\CL}\leq C\|f\|_{\M(\RN)}.$
The proof of Lemma~\ref{le2.3} is complete.
\end{proof}

\begin{lemma}\label{le2.2}
Let $1\leq p<\infty$ and $0<\lambda<n$.  Suppose $f\in \CL$. Then:
  \begin{itemize}
\item[(i)] For any $t>0$ and $K>1$, there exists  $C>0$ independent of $t$ and $K$ such that
\begin{align*}
\|e^{-tL}f(x)-e^{-KtL}f(x)\|_{L^\infty(\RN)}\leq C\ t^{\frac{\lambda-n}{pm}}\big\|f\big\|_{\CL}.
\end{align*}

\item[(ii)]For any $\delta>0$, there exists $C(\delta)>0$ such that

\begin{align*}
\int_{\mathbb{R}^n} \frac{|e^{-tL}f(x)-f(x)|^p}{\left(t^{1/m}+|x|\right)^{n+\delta}} \,dx\leq C(\delta)
t^{-(n-\lambda+\delta)/m}\big\|f\big\|_{\CL}^p.
\end{align*}
\end{itemize}
\end{lemma}
\begin{proof}
For the proof of (i), we refer to (i) of \cite[Lemma ~3]{DXY}.
The proof of (ii) can be obtained by making minor modifications with that of
(ii) of \cite[Lemma ~3]{DXY}, and we skip it.
\end{proof}

\begin{lemma}\label{le3.2}
Let $1\leq p<\infty$ and $0<\lambda<n$. For every $f\in \CL$, there exists a function
$  \sigma_L(f)\in \mathcal{M}_p$
such that $\lim\limits_{t\to +\infty} \|e^{-tL}f-\sigma_L(f)\|_{L^\infty(\RN)}=0$.
 Moreover, $\sigma_L(f)\in \mathcal{K}_{L,p}$.
\end{lemma}

\begin{proof}
Let $f\in \CL$. For any $0<t<s$, we apply (i) of Lemma \ref{le2.2} to obtain
\begin{align*}
|e^{-tL}f(x)-e^{-sL}f(x)|\leq C \, t^{\frac{\lambda-n}{pm}}\|f\|_{\CL},
\end{align*}
for some constant  $C$   independent of  $t, s$ and $x$, and thus
$|e^{-tL}f(x)-e^{-sL}f(x)| \to 0$  uniformly   as   $t\to +\infty$. By the completeness of $\mathbb{R}^n$, there
exists a function $\sigma_L(f)(x)=\lim\limits_{t\to +\infty}  e^{-tL}f(x) $  such that
\begin{align}\label{e3.1}
\lim\limits_{t\to +\infty} \|e^{-tL}f-\sigma_L(f)\|_{L^\infty(\RN)}=0.
\end{align}

Let us  prove  that $\sigma_L(f)\in \mathcal{M}_p$. Set
\begin{align*}
{\beta_0}:=\left \{
\begin{array}{ll}
\Theta(L),\ \ \ \ \ \  \  \quad \quad \quad \qquad \qquad {\rm if } \ {\Theta }(L) <\infty;\\\\
 {\rm any \ positive \ number },\ \ \ \ \ \ \ \ \ \ \ {\rm if} \ {\Theta }(L)=\infty.
\end{array}
\right.
\end{align*}
Using this notation, we have that
$f\in \mathcal{M}_{p, \beta}$, for any $\beta\in(0,\beta_0)$.
It suffices to show that
\begin{align}\label{e3.2}
\int_{{\mathbb R}^n} \frac{|\sigma_Lf(x)|^p}{(1+|x|)^{n+{\beta}}}\, dx
 <\infty.
\end{align}
To prove \eqref{e3.2}, we   choose  some   $t_0>1$ in \eqref{e3.1}   large enough  such that
 $\|e^{-t_0L}f-\sigma_L(f)\|_{L^\infty(\RN)}\leq C.$ One writes
$$|\sigma_Lf(x)|\leq  |\sigma_Lf(x)-e^{-t_0L}f(x)| +|e^{-L}f(x)-e^{-t_0L}f(x)| + |f(x)-e^{-L}f(x)| +|f(x)|.
$$
From (i) of Lemma~\ref{le2.2},
$\|e^{-L}f(x)-e^{-t_0L}f(x)\|_{L^\infty(\RN)}\leq C \  \big\|f\big\|_{\CL}.$ This, in combination with
(ii) of Lemma~\ref{le2.2} and the fact that  $f\in \mathcal{M}_{p, \beta}$, shows
\begin{align} \label{e3.3}
\int_{{\mathbb R}^n} \frac{|\sigma_Lf(x)|^p}{(1+|x|)^{n+{\beta}}}\, dx&\leq  C+ C\big\|f\big\|_{\CL}
 + C\int_{{\mathbb R}^n} \frac{|f(x)|^p}{(1+|x|)^{n+{\beta}}}\, dx <\infty,
\end{align}
and hence
 $\sigma_Lf\in \mathcal{M}_p.$

Finally, we  apply \eqref{e3.1} to obtain that  for any $s>0$,
\begin{align*}
e^{-sL}\sigma_L{f}(x)&=e^{-sL}\lim\limits_{t\to +\infty} e^{-tL}f(x)
 =\lim\limits_{t\to +\infty}e^{-sL}e^{-tL}f(x)
 =\lim\limits_{t\to +\infty}e^{-(s+t)L}f(x)=\sigma_L{f}(x)
\end{align*}
for almost all $x\in\RN$, which yields that  $\sigma_L(f)\in \mathcal{K}_{L,p}$.
This completes  proof of Lemma~\ref{le3.2}.
\end{proof}

\bigskip

\noindent {\it Proof of Theorem ~\ref{th3.1}.}  Suppose that $f\in \M(\mathbb{R}^n)$.
From  Lemma \ref{le2.3}, we have that
  $f\in \CL  $ and $ \|f\|_{\CL}\leq C\|f\|_{\M(\RN)}.
 $ Next we prove that
\begin{align}\label{e3.4}
f-\sigma_L(f)\in \CL   \quad  {\rm and } \quad  \|f-\sigma_L(f)\|_{\CL}\leq C\|f\|_{\M(\RN)},
\end{align}
and it reduces to show  that $\sigma_L(f)(x)=0$, a.e. $x\in\RN$.
For any $t>0$ and almost everywhere $x\in \RN$, we use Lemma~\ref{le2.3}
  to obtain
\begin{align*}
|e^{-t^mL}(f)(x)|
  \leq Ct^{(\lambda-n)/p}\|f\|_{\M}.
\end{align*}
Since $\lambda\in (0, n)$, we have
$$
|\sigma_L(f)(x)|\leq \lim_{t\to +\infty}|e^{-tL}(f)(x)|
=0,  \quad {\rm for \ a.e.\ } x\in \mathbb{R}^n
$$
and \eqref{e3.4} holds.

For the converse part of Theorem~\ref{th3.1}, we will show that for any $f\in \CL$,
  there exists constant $C>0$  such that for any ball $B\subset \mathbb{R}^n$,
\begin{align}\label{e3.11}
\Big(\frac{1}{r_B^{\lambda}}\int_B |f(x)-\sigma_L(f)(x)|^p \, dx\Big)^{1/p}\leq C \,\|f\|_{\CL}.
\end{align}
Indeed, one can write
\begin{align}\label{e3.12}
 \left(\frac{1}{r_B^{\lambda}}\int_B |f(x)-\sigma_L(f)(x)|^p \, dx\right)^{1/p}
\leq &\|f\|_{\CL}
+\left(\frac{1}{r_B^{\lambda}}\int_B |e^{-r_B^mL}f(x)-\sigma_L(f)(x)|^p \, dx\right)^{1/p}.
\end{align}
By Lemmas~\ref{le2.2} and ~\ref{le3.2},
\begin{eqnarray*}
\|e^{-r_B^mL}f(x)-\sigma_L(f)(x)\|_{L^\infty(\RN)}&\leq& \lim_{k\to +\infty}\|e^{-r_B^mL}f(x)
-e^{-kr_B^mL}f(x)\|_{L^\infty(\RN)} \\
&+&\lim_{k\to +\infty}\|
 e^{-kr_B^mL}f(x)-\sigma_L(f)(x)\|_{L^\infty(\RN)}\\
&\leq& Cr_B^{\lambda-n\over p}\|f\|_{\CL}.
\end{eqnarray*}
This, in combination with  \eqref{e3.12}, shows the desired estimate \eqref{e3.11}.
Hence, the proof of Theorem~\ref{th3.1} is complete.

\medskip

\subsection{Campanato spaces associated to Schr\"odinger operators and classical Morrey spaces are equivalent}

 In this subsection, let us consider the Schr\"odinger operators as (\ref{e4.1})
\begin{align*}
L=-\Delta+V(x) \qquad  {\rm on} \  \RN, \ \quad  n\geq 3,
\end{align*}
 where $V\in L^1_{\rm loc}(\RN)$ and $V\not\equiv 0$ is  a non-negative function in
$ B_q$ for some $q\geq n.$
For such Schr\"odinger operator $L$,
   the semigroup kernels ${\mathcal P}_t(x,y)$  of the operators $e^{-t\sqrt{L}}$ satisfy the following properties:

\begin{lemma}\label{size}   For  every $N>0$, there exists
 a constant $C=C_{N}$ such that for $x,y\in \RR$ and $t>0,$

\begin{align}\label{e4.3}
 \quad | {\mathcal P}_t(x,y)|\leq C{t \over (t^2+|x-y|^2)^{n+1\over 2}} \left(1+ {(t^2+|x-y|^2)^{1/2}\over \rho(x)}+{(t^2+|x-y|^2)^{1/2}\over \rho(y)}\right)^{-N};
 \end{align}

\begin{align}\label{e4.4}
 \quad  |t\nabla {\mathcal P}_t(x,y)|\leq C {t  \over (t^2+|x-y|^2)^{n+1\over 2}} \left(1+ {(t^2+|x-y|^2)^{1/2}\over \rho(x)}
+{(t^2+|x-y|^2)^{1/2}\over \rho(y)}\right)^{-N}.
\end{align}
\noindent Here,
\begin{equation*}
 \rho(x)=\sup \Big{\{}r>0: \ {1\over r^{n-2}} \int_{B(x,r)} V(y)dy \leq 1 \Big{\}}.
\end{equation*}
\end{lemma}
\begin{proof}
 For the proof,  see Proposition 3.6, \cite{MSTZ}.
\end{proof}

 As an application of  Theorem \ref{th3.1},  we have  the following result.

\begin{theorem}\label{th4.1}  For $0<\lambda<n, 1\leq p<\infty$,  there holds
$$\mathscr{L}_L^{p,\lambda}(\mathbb{R}^n)= \mathscr{L}_{\sqrt{L}}^{p,\lambda}(\mathbb{R}^n)
= {L}^{p,\lambda}(\mathbb{R}^n).
$$
\end{theorem}

\begin{proof} By  Theorem~\ref{th3.1},  it suffices to show that for every $1\leq p<\infty,$
\begin{eqnarray}\label{e4.6}
\mathcal{K}_{L, p}= \mathcal{K}_{\sqrt{L}, p} =\{0\}.
\end{eqnarray}
 The proof of \eqref{e4.6} can be obtained by making modifications
 with Proposition 6.5, \cite{DY2}. See also \cite{Sh}. We give a brief argument
of this proof for completeness and convenience for the reader.

 Let us  prove that $\mathscr{L}_{\sqrt{L}}^{p,\lambda}(\mathbb{R}^n)
= {L}^{p,\lambda}(\mathbb{R}^n)$. For any $d\geq 0$,  one writes
\begin{eqnarray*}
{ {\mathcal H}_d(L)}=\Big\{f\in W^{1,2}_{\rm loc}({\mathbb R}^n):
L f =0 \ {\rm and}\   \ |f(x)|=O(|x|^d) \ \ {\rm as}\ |x|\rightarrow \infty\Big\}
\end{eqnarray*}
and
\begin{eqnarray*}
{ {\mathcal H}_{L}}=\bigcup_{d:\ 0\leq d<\infty}
{ {\mathcal H}_d(L)}.
\end{eqnarray*}
By  Lemma 2.7 of \cite{DYZ}, it follows that
for any $d\geq 0$,
$$
{ {\mathcal H}_{L}}={ {\mathcal H}_d(L)}=\big\{0\big\}.
$$

 Assume that $f\in
\mathcal{K}_{\sqrt{L}, p}\cap \mathcal{M}_p$. From estimates \eqref{e4.3} and \eqref{e4.4}, we see that
$f=e^{-t\sqrt{L}}f\in   W^{1,2}_{\rm
loc}({\mathbb R}^n)$  and $|f(x)|=O(|x|^{\epsilon}) $ for some $\epsilon>0.$
Because of the growth of $f$, we use a standard approximation argument
through a sequence $f_k$ as follows. For any $k\in{\mathbb N}$,
we denote by
$\eta_k$    a standard $C^{\infty}$ cut-off function
which is $1$ inside the ball $B(0, k)$, zero outside
$B(0, k+1)$, and let $f_k=f\eta_k\in
W^{1,2}({\mathbb R}^n)$. Since
$f=e^{-t\sqrt{L}}f$, we have that for any $\varphi\in C^1_0({\mathbb R}^n)$,
\begin{eqnarray*}
\langle Lf, \varphi\rangle = \langle Le^{-t\sqrt{L}}f, \varphi\rangle
&=&\lim_{k\rightarrow \infty}\langle Le^{-t\sqrt{L}}f_k, \varphi\rangle\\
&=&\lim_{k\rightarrow \infty}\langle {d^2\over dt^2}e^{-t\sqrt{L}}f_k,
\varphi\rangle
=\langle {d^2\over dt^2}e^{-t\sqrt{L}}f,
\varphi\rangle\\
&=&\langle {d^2\over dt^2}f, \varphi\rangle
=0,
\end{eqnarray*}
which  proves $f\in {\mathcal H}_{\epsilon}(L)=\{0\}$. Hence,  $\mathcal{K}_{\sqrt{L}, p}=\{0\}.$
Similar argument above shows that $\mathcal{K}_{L, p}=\{0\}.$
This completes the proof of Theorem~\ref{th4.1}.
\end{proof}

\begin{cor} Let $L=-\Delta+V$, where $V\not\equiv 0$ is  a non-negative potential in
$ B_q$ for some $q\geq n.$ From Theorem~\ref{th4.1} and the previous results
proved in \cite{DY2, DGMTZ, MSTZ}, we have the following result:
 \begin{itemize}
\item[(i)]  $\mathscr{L}_L^{p,\lambda}(\mathbb{R}^n)  = L^{p,\lambda}(\RN)$, when $\lambda\in (0, n);$

\item[(ii)]  $\mathscr{L}_L^{p,\lambda}(\mathbb{R}^n)\subsetneqq {\rm BMO}(\RN)$, when $\lambda=n$ (see \cite{DY2,DGMTZ});

\item[(iii)]  $\mathscr{L}_L^{p,\lambda}(\mathbb{R}^n)
\subsetneqq {\rm Lip}_{\alpha}({\mathbb R}^n)$ with $\alpha=(\lambda-n)/p$,
 when $\lambda\in (n, n+p)$ (see \cite{MSTZ}).
 \end{itemize}
\end{cor}

\medskip

\section{Proof of Theorem \ref{th4.2}}

\noindent {\it Proof of (1) of Theorem \ref{th4.2}.}
 Suppose $f\in L^{2,\lambda}(\RN)$.
From  Lemmas 2.4 and 3.9 of \cite{DYZ}, it follows that $u(x,t)=e^{-t\sqrt{L}}f(x)\in C^1(\mathbb{R}^{n+1}_+)$. It will be enough if we have proved
$$
 \|u\|_{{\rm HL}^{2,\lambda}_L(\Real_+^{n+1})}\leq C \|f\|_{L^{2,\lambda}(\RN)}.
 $$
 In fact,
\begin{eqnarray*}
\left(\frac{1}{r_B^\lambda} \int_0^{r_B}\!\int_B \big |t\nabla e^{-t\sqrt{L}}f(x)\big |^2
\, dx\frac{dt}{t}\right)^{1/2}
&\leq& \sum_{k=0}^{\infty}\frac{1}{r_B^{\lambda/2}}\left(\int_0^{r_B}\!\int_B
\big |t\nabla e^{-t\sqrt{L}}f_k(x)\big |^2 \, dx\frac{dt}{t}\right)^{1/2}\\
&=&:\sum_{k=0}^{\infty}J_k,
\end{eqnarray*}
where $f_0=f\chi_{2B}$ and $f_k=f\chi_{2^{k+1}B\backslash 2^kB}$ for $k=1,2, \cdots.$
Obviously,
$
|J_0|\leq   C \|f\|_{L^{2,\lambda}(\RN)}.
$
For any $x\in B$ and $k\in \mathbb{N}$,  we apply  estimates \eqref{e4.3} and \eqref{e4.4}
  to obtain
\begin{align*}
\big |t\nabla e^{-t\sqrt{L}}f_k(x)\big | &\leq  C\int_{2^{k+1}B\backslash 2^kB}
 \frac{t}{(t+|x-y|)^{n+1}}  |f(y)|\, dy\nonumber\\
&\leq C \frac{t}{(2^kr_B)^{n+1}} \int_{2^{k+1}B}|f(y)| \,dy \nonumber\\
&\leq C \frac{t}{(2^kr_B)^{1+\frac{n-\lambda}{2}}} \|f\|_{L^{2,\lambda}(\RN)},
\end{align*}
which yields
\begin{align*}
|J_k|  \leq C 2^{-k(1+\frac{n-\lambda}{2})} \|f\|_{L^{2,\lambda}(\RN)}.
\end{align*}
Hence,
$
\sum_{k=1}^\infty  |J_k|\leq  C\|f\|_{L^{2,\lambda}(\RN)},
$
and then
$\|u\|_{{\rm HL}^{2,\lambda}_L(\Real_+^{n+1})}\leq C \|f\|_{L^{2,\lambda}(\RN)}$.

\bigskip

Before we give the proof of (2) of Theorem \ref{th4.2}, we need some auxiliary lemmas.

\begin{lemma}\label{le2.6} Suppose $0\leq V\in L^q_{\rm loc}(\RN)$ for some $q> n/2.$
Let $u$ be a weak solution of ${\mathbb L}u=0$
in
the ball $B(Y_0,2r) \subset \mathbb{R}^{n+1}$.
Then for any $p\geq 1$,   there exits a constant $C=C(n, p)>0$ such that

\begin{eqnarray*}
\sup_{B(Y_0, r)}| u(Y)| \leq C \Big({1\over r^{n+1}}
\int_{B(Y_0, 2r)}| u(Y)|^pdY\Big)^{1/p}.
\end{eqnarray*}
\end{lemma}
\begin{proof}
For the proof, we refer to Lemma 2.6 of \cite{DYZ}.
\end{proof}
\begin{lemma}\label{le2.7}
Suppose $ V\in B_q$ for some $q\geq n/2.$
 Assume that  $u\in W_{\rm loc}^{1,2}(\RN)$ is a weak solution of  $Lu=(-\Delta+V)u=0$ in $\RN$. Also assume that
there is a $d>0$ such that
\begin{eqnarray*}
\int_{\RR}{|u(x)|\over 1+|x|^{n+d}} dx\leq C_{d}<\infty.
\end{eqnarray*}
Then $u=0$ in $\RN$.
\end{lemma}
\begin{proof}
For the proof, we refer to Lemma 2.7 of \cite{DYZ}.
\end{proof}

\begin{lemma}\label{le3.1} Let $0<\lambda<n$. For every  $u\in {{\rm HL}^{2,\lambda}_L}({\mathbb R}_+^{n+1})$ and
 for every $k\in{\mathbb N}$, there exists a constant $C(k,n,\lambda)>0$ such that
\begin{equation*} \label{dd}
\int_{\RN}{|u(x,{1/k})|^2\over (1+|x|)^{2n}}  dx\leq C(k,n,\lambda) <\infty,
\end{equation*}
hence $u(x, 1/k)\in L^2((1+|x|)^{-2n}dx)$. Therefore for  all $k\in{\mathbb N}$, $e^{-t\sqrt{L}}(u(\cdot, {1/k}))(x)$ exists
everywhere in ${\mathbb R}^{n+1}_+$.
\end{lemma}

\begin{proof}  Since $u\in C^{1}({\mathbb R}^{n+1}_+)$, it  reduces to show that for every $k\in{\mathbb N},$
\begin{eqnarray}\label{e3.1}
\int_{|x|\geq 1} {|u(x,{1/k})- u(x/|x|, 1/k) |^2 \over (1+|x|)^{2n}} dx\leq C(k,n,\lambda)\|u\|^2_{{{\rm HL}^{2,\lambda}_L}({\mathbb R}_+^{n+1})}<\infty.
 \end{eqnarray}
To do this, we write
\begin{align*}
&\hspace{-0.3cm}u(x, 1/k)- u(x/|x|, 1/k)\\&=\big[u(x, 1/k)- u(x, |x|)\big]
 +\big[u(x, |x|)-u(x/|x|, |x|)\big] +\big[u(x/|x|, |x|)-u(x/|x|, 1/k)\big].
 \end{align*}
Let
  \begin{eqnarray*}
 I=  \int_{|x|\geq 1 } {|u(x, 1/k)- u(x, |x|) |^2 \over (1+|x|)^{2n}} dx,
 \end{eqnarray*}
  \begin{eqnarray*}
 II=  \int_{|x|\geq 1 } {|u(x, |x|)-u(x/|x|, |x|) |^2 \over (1+|x|)^{2n}} dx,
 \end{eqnarray*}and
  \begin{eqnarray*}
 III=  \int_{|x|\geq 1 } {|u(x/|x|, |x|)-u(x/|x|, 1/k) |^2 \over (1+|x|)^{2n}} dx.
 \end{eqnarray*}

 Set $Y_0=(x,t)$ and $r=t/4$. We use  Lemma~\ref{le2.6} for $\partial_t u$ and Schwarz's inequality to obtain
 \begin{eqnarray}\label{e3.2}
 \big| \partial_t u(x, t)\big| &\leq& C\Big({1\over r^{n+1}} \int_{B(Y_0, 2r)}
 | \partial_s u(y, s) |^2 {dY } \Big)^{1/2}\nonumber\\
  &\leq& C\Big({1\over t^{n+1}} \int_{B(x, t/2)}\int_{t/2}^{3t/2}
 | \partial_s u(y, s) |^2 {dsdy } \Big)^{1/2}\nonumber\\
  &\leq& C t^{-1}\Big({1\over |B(x, 2t)|} \int_{0}^{2t}\int_{B(x, 2t)}
  s|  \partial_s u(y, t) |^2 {dyds } \Big)^{1/2}
  \nonumber\\
  &\leq&  C  t^{-(1+\frac{n-\lambda}{2})}\|u\|_{{{\rm HL}_L^{2,\lambda}} (\Real_+^{n+1})}.
  \end{eqnarray}
Thus,
 \begin{eqnarray}\label{e3.3}
 | u(x, |x|)-u(x, 1/k) |  =  \Big| \int_{1/k}^{|x|}   \partial_t u(x, t) dt \ \Big| \leq \frac{2}{n-\lambda} k^{(n-\lambda)/2} \|u\|_{{{\rm HL}_L^{2,\lambda}}(\Real_+^{n+1})}.
 \end{eqnarray}
It follows that
   \begin{eqnarray*}
 I+III
  &\leq&  C(k,n,\lambda)\|u\|^2_{{{\rm HL}_L^{2,\lambda}}(\Real_+^{n+1})} \int_{|x|\geq 1 } {1\over (1+|x|)^{2n}}  dx  \\
 &\leq& C(k,n,\lambda)  \|u\|^2_{{{\rm HL}_L^{2,\lambda}}(\Real_+^{n+1})}.
  \end{eqnarray*}

For the term $II,$  we  have that for any $x\in \RR,$
 $$
  u(x, |x|)- u(x/|x|, |x|)  =\int_1^{|x|}  D_r u(r\omega, |x|)  dr, \ \ \ \ x=|x| \omega.
 $$
 Let $B=B(0, 1)$  and  $2^mB=B(0, 2^m)$.
Note that for every $m\in{\mathbb N}$,  we have
  \begin{align*}
    \int_{2^mB\backslash 2^{m-1}B} \left| \int_1^{|x|}\left|    D_r  u(r\omega, |x|) \right|   dr \right|^2 dx
  &=     \int_{2^{m-1}}^{2^{m}} \int_{|\omega|=1} \left| \int_1^{\rho}     D_r u(r\omega, \rho)    dr \right|^2 \rho^{n-1}  d\omega d\rho   \nonumber\\
      &\leq   2^{mn} \int_{2^{m-1}}^{2^{m}} \int_{|\omega|=1}\int_1^{ 2^{m}}    |  D_r u(r\omega, \rho)  |^2 dr d\omega d\rho   \nonumber\\
   	  &\leq    2^{mn} \int_{2^{m-1}}^{2^{m}} \int_{2^mB\backslash B}   |  \nabla_y u(y, t)  |^2 |y|^{1-n} dy dt\nonumber\\
	  &\leq    2^{mn} \int_{2^{m-1}}^{2^{m}} \int_{2^mB}   |  \nabla_y u(y, t)  |^2  dy dt,
  \end{align*}
  which gives
  \begin{align*}
 \int_{2^mB\backslash 2^{m-1}B}   | u(x, |x|) - u(x/|x|, |x|) |^2    dx
      &\leq     C2^{m(2n-1)} \left( {1\over |2^mB|}  \int_{0}^{2^{m}} \int_{2^mB }    | t \nabla_y u(y, t)  |^2
	{dydt\over t}   \right) \nonumber
	 \\
    &\leq  C  2^{m(n+\lambda-1)}\|u\|^2_{{{\rm HL}_L^{2,\lambda}}(\Real_+^{n+1})}.
  \end{align*}
Therefore,
 \begin{align*}
II
  &\leq C \sum_{m=1}^{\infty}  {1\over 2^{2mn}}
   \int_{2^mB\backslash 2^{m-1}B}    | u(x, |x|) - u(x/|x|, |x|) |^2    dx\leq C(n,\lambda)\|u\|^2_{{{\rm HL}_L^{2,\lambda}}(\Real_+^{n+1})}.
  \end{align*}
  Combining estimates of $I, II$ and $III$, we have obtained \eqref{e3.1}.

  Note that by Lemma~\ref{size},  the semigroup kernels ${\mathcal P}_t(x,y)$, associated to $e^{-t\sqrt{L}}$,
  decay
faster than any power of $1/|x-y|$.
Hence,    for  all $k\in{\mathbb N}$, $e^{-t\sqrt{L}}(u(\cdot, {1/k}))(x)$ exists
everywhere in ${\mathbb R}^{n+1}_+$.
This completes the proof.
  \end{proof}

\medskip

\noindent {\it Proof of   (2) of Theorem \ref{th4.2}.} To prove it, we will adopt the argument as in \cite{FJN, DYZ,JXY} and apply the key Theorem \ref{th4.1}.
Suppose
 $u\in {\rm HL}^{2,\lambda}_L(\Real_+^{n+1})$. Our aim is to look for a
 function $f\in L^{2,\lambda}(\RN)$ such that
$$
u(x,t)=e^{-t\sqrt{L}}f(x),  \quad \quad  {\rm for \ each } \quad (x,t)\in {\mathbb R}^{n+1}_+.
$$

To do this, for every $k\in {\mathbb N}$, we write $f_k(x)=u(x, k^{-1})$. We will first show that
\begin{align}\label{e4.10}
\sup_{k\in \mathbb{N}}\|f_k\|_{{\mathscr L}_{\sqrt{L}}^{2,\lambda}(\RN)}
\leq C\|u\|_{{\rm HL}^{2,\lambda}_L(\Real_+^{n+1})}.
\end{align}
If (\ref{e4.10}) has been proved, then by { Theorem~\ref{th4.1}}: $L^{2,\lambda}(\RN)= {\mathscr L}_{\sqrt{L}}^{2,\lambda}(\RN)$, we can obtain
\begin{align}\label{e4.11}
\|f_k\|_{L^{2,\lambda}(\RN)}=\|f_k\|_{{\mathscr L}_{\sqrt{L}}^{2,\lambda}(\RN)}
\leq C\|u\|_{{\rm HL}^{2,\lambda}_L(\Real_+^{n+1})}<\infty.
\end{align}
Since Lemmas \ref{le2.6},\ref{le2.7}, \ref{le3.1} are at our disposal, we can follow the arguments of Lemmas 3.2 and 3.4 of
\cite{DYZ} to obtain  the following two facts of function
  $u $ in the space  $  {\rm HL}^{2,\lambda}_L(\Real_+^{n+1})$ with $0<\lambda<n$:

1) For every $k\in{\mathbb N}$,
$
u(x, t+{1/k})=e^{-t\sqrt{L}}\big(u(\cdot, {1/k})\big)(x)$ for all $ x\in\RN $ and $ t>0.
$

2) There exists a constant $C>0$ (depending only on $n$ and $\lambda$) such that for all  $k\in{\mathbb N},$
$$
\sup_{x_B, r_B}   r_B^{-\lambda}\int_0^{r_B}\!\int_{B(x_B, r_B)} t
\big| \partial_t e^{-t\sqrt{L}}\big(u(\cdot, {1/k})\big)(x)\big|^2 {dx dt}
  \leq C\|u\|^2_{{\rm HL}^{2,\lambda}_L(\Real_+^{n+1})}<\infty.
$$

\noindent By Lemma 3.5 of \cite{DYZ},  we have
 \begin{align*}
\|f_k\|_{{\mathscr L}_{\sqrt{L}}^{2,\lambda}(\RN)}\leq C\sup_B
\left(\frac{1}{r_B^\lambda} \int_0^{r_B} \!\!\int_B|t\partial_t e^{-t\sqrt{L}}
f_k(x)|^2 \frac{dxdt}{t} \right)^{1/2}\leq C\|u\|_{{\rm HL}^{2,\lambda}_L(\Real_+^{n+1})}<\infty,
\end{align*}
which implies \eqref{e4.10}.

Next, we will look for    a function $f\in L^{2,\lambda}$ through
$L^2(B(0,2^j))$-boundedness of $\{f_k\}$ for each $j\in {\mathbb N}$.
Indeed, for every  $j\in \mathbb{N}$  we use \eqref{e4.11} to obtain
$$
\int_{B(0,2^j)} |f_k(x)|^2 \leq C 2^{j\lambda}
\|u\|_{{\rm HL}^{2,\lambda}_L(\Real_+^{n+1})}^2.
$$
This tells us that  the sequence $\{f_k\}_{k=1}^\infty$ is bounded in $L^2(B(0,2^j))$.
So, after passing to a subsequence, the sequence converges weakly to  a function $g_j\in L^2(B(0,2^j))$.
Now we define a function $f(x)$ by
$$
f(x)=g_j(x), \quad \quad  {\rm if } \ x\in  B(0,2^j),  \quad j=1,2,\cdots.
$$
It is easy to see that $f$ is well defined on $\RN=\cup_{j=1}^{\infty}B(0, 2^j).$
Also  we can  check that for any open ball $B\subset \RN$,
$$
\int_{B} |f(x)|^2 dx\leq C r_B^{\lambda} \|u\|_{{\rm HL}^{2,\lambda}_L(\Real_+^{n+1})}^2,
$$
which implies
$$
\|f\|_{L^{2,\lambda}(\RN)}\leq C \|u\|_{{\rm HL}^{2,\lambda}_L(\Real_+^{n+1})}.
$$

Finally,  we will show that $u(x,t)=e^{-t\sqrt{L}} f(x)$.  Since $u(x,\cdot)$ is continuous on ${\mathbb R}_+$, we have
 $
u(x,t)=\lim\limits_{k\to +\infty}u(x,t+k^{-1}),
 $
and by 1),
$u(x,t)=\lim\limits_{k\to +\infty}e^{-t\sqrt{L}}\big(u(\cdot,k^{-1})\big)(x).
$
It reduces to show
 \begin{align}
 \label{e4.12}
 \lim\limits_{k\to +\infty}e^{-t\sqrt{L}}\big(u(\cdot,k^{-1})\big)(x)=e^{-t\sqrt{L}}f(x).
\end{align}
Indeed, we  recall that  ${\mathcal P}_t(x,y)$ is the kernel of $e^{-t\sqrt{L}}$, and for any $\ell\in \mathbb{N}$,  we write
\begin{align*}
e^{-t\sqrt{L}}\big(u(\cdot,k^{-1})\big)(x)&=\int_{B(x,2^\ell t)} {\mathcal P}_t(x,y) f_k(y) dy
+ \int_{\RN\backslash B(x,2^\ell t)} {\mathcal P}_t(x,y) f_k(y) dy.
\end{align*}
Using   the H\"older  inequality, we obtain
\begin{align*}
\left|\int_{\RN\backslash B(x,2^\ell t)} {\mathcal P}_t(x,y) f_k(y) dy\right|
&\leq C\sum_{i=\ell}^{\infty} 2^{-i} (2^it)^{-n}\int_{B(x, 2^{i+1}t)} |f_k(y)| \, dy\\
&\leq C\sum_{i=\ell}^{\infty} 2^{-i} (2^it)^{(\lambda-n)/2} \|f_k\|_{L^{2,\lambda}(\RN)}\\
&\leq  C 2^{-\ell(1+\frac{n-\lambda}{2})}   t^{(\lambda-n)/2} \|f_k\|_{L^{2,\lambda}(\RN)}.
\end{align*}
From \eqref{e4.11}, we have that $ \|f_k\|_{L^{2,\lambda}(\RN)}\leq C \|u\|_{{\rm HL}^{2,\lambda}_L(\Real_+^{n+1})}$
for some constant $C>0$ independent of $k.$
Hence,
$$
\lim\limits_{\ell\to +\infty}\lim\limits_{k\to +\infty}\int_{\RN\backslash B(x,2^\ell t)} {\mathcal P}_t(x,y) f_k(y) dy
= \lim\limits_{\ell\to +\infty}\left(  C 2^{-\ell (1+\frac{n-\lambda}{2})}   t^{(\lambda-n)/2}
\|u\|_{{\rm HL}^{2,\lambda}_L(\Real_+^{n+1})}
\right)=0
$$
since $\lambda\in (0, n).$ Therefore,
\begin{align}
 \label{e4.13}
 \lim\limits_{k\to +\infty}e^{-t\sqrt{L}}\big(u(\cdot,k^{-1})\big)(x)=\lim\limits_{\ell\to +\infty}\lim\limits_{k\to +\infty}
 \int_{B(x,2^\ell t)} {\mathcal P}_t(x,y) f_k(y) dy =
 e^{-t\sqrt{L}}f(x),
\end{align}
and \eqref{e4.12}
follows readily.
Then we  have showed that $u(x,t)=e^{-t\sqrt{L}} f(x)$. The proof of Theorem \ref{th4.2} is complete.

\bigskip

\noindent
{\bf Acknowledgments.} \  L. Song is supported in part by Guangdong Natural Science Funds for Distinguished Young Scholar
(No. 2016A030306040) and  NNSF of China (Nos 11471338 and 11622113). L. Yan is supported by the NNSF of China (Nos 11371378 and  ~11521101)  and Guangdong Special Support Program.


\medskip

\begin{thebibliography}{99999}




\bibitem{AX2} D. Adams and J. Xiao,  Morrey spaces in harmonic analysis.
{\it Ark. Mat.},  {\bf 50}  (2012),  no. 2, 201--230.






\bibitem{Ca} S. Campanato, Propriet¨¤ di una famiglia di spazi funzionali.
(Italian) {\it Ann. Scuola Norm. Sup. Pisa}, {\bf 18} (1964), 137--160.




\bibitem {DDSY} D.G. Deng, X.T. Duong, A. Sikora and L.X. Yan,  Comparison of the classical
BMO with the BMO spaces associated with operators and applications.
{\it Rev. Mat. Iberoam.}, {\bf 24} (2008), no. 1, 267--296.





\bibitem {DXY} X.T. Duong, J.  Xiao and L.X. Yan, Old and new Morrey spaces with heat
kernel bounds. {\it J. Fourier Anal. Appl.}, {\bf 13} (2007), no. 1, 87--111.




\bibitem {DY1} X.T. Duong and L.X. Yan, New function spaces of BMO type,
John-Nirenberg inequality, interpolation and applications. {\it Comm.
Pure Appl. Math.}, {\bf 58} (2005), 1375--1420.




\bibitem {DY2} X.T. Duong and L.X. Yan,  Duality of Hardy and BMO spaces
associated with operators with heat kernel bounds. {\it J. Amer. Math.
Soc.}, {\bf 18} (2005), 943--973.




\bibitem {DYZ} X.T. Duong, L.X. Yan and C. Zhang, On characterization of Poisson integrals
of Schr\"odinger operators with BMO traces. {\it J. Funct. Anal.},
{\bf 266} (2014), no. 4, 2053--2085.


\bibitem {DGMTZ}J. Dziuba\'nski, G.  Garrig\'os,T. Mart\'inez, J.L. Torrea and J. Zienkiewicz,
 BMO spaces related to Schr\"odinger operators with potentials satisfying a reverse H\"older inequality. {\it Math. Z.}, {\bf 249} (2005), no. 2, 329--356.

\bibitem {FJN} E.B. Fabes, R.L. Johnson and U. Neri,
Spaces of harmonic functions representable by Poisson integrals of functions
in BMO and ${\mathcal L}_{p,\lambda}$.
{\it Indiana Univ. Math. J.}, {\bf 25} (1976), no. 2, 159--170.




\bibitem {FS} C. Fefferman and E.M. Stein, $H^p$ spaces of several
variables. {\it  Acta Math.}, {\bf 129} (1972), 137--195.




\bibitem{JTW} S. Janson, M.H. Taibleson, G. Weiss,
Elementary characterizations of the Morrey-Campanato spaces. {\it Lecture Notes
in Math.},  {\bf 992} (1983), 101--114.



\bibitem  {JN} F. John and L. Nirenberg, On functions of bounded
mean oscillation. {\it Comm. Pure Appl. Math.},
{\bf 14} (1961), 415--426.


\bibitem {JXY} R.J. Jiang, J. Xiao and D.C. Yang, Towards spaces of harmonic
functions with traces in square Campanato space and their scaling invariant. {\it Anal. Appl. (Singap.)}, {\bf 14} (2016), no. 5, 679--703.


\bibitem{MSTZ} T.  Ma, P.  Stinga, J.  Torrea and C. Zhang,
{Regularity properties of Schr\"odinger operators}.
{\it J. Math. Anal. Appl.}, {\bf 388} (2012), 817--837.





\bibitem{Mo} C.B. Morrey, On the solutions of quasi-linear elliptic partial differential equations.
{\it Trans. Amer. Math. Soc.}, {\bf 43} (1938), no. 1, 126--166.



\bibitem{Na} E. Nakai, The Campanato, Morrey and H\"older spaces on spaces of homogeneous
type. {\it Studia Math.}, {\bf 176} (2006), no. 1, 1--19.






\bibitem{P} J. Peetre, On the theory of $\mathcal{L}_{p,\lambda}$ spaces.
{\it J. Funct. Anal.}, {\bf 4} (1969), 71--87.


\bibitem {RSS}
H. Rafeiro, N. Samko and S. Samko,
Morrey-Campanato spaces: an overview.  {\it Operator theory, pseudo-differential equations,
 and mathematical physics}, 293--323,
Oper. Theory Adv. Appl., {\bf 228}, Birkh\"auser/Springer Basel AG, Basel, 2013.

\bibitem{Shen} Z.W. Shen,
{$L^p$ estimates for Schr\"odinger operators with certain potentials}.
\textit{Ann. Inst. Fourier (Grenoble)}, \textbf{45} (1995), 513--546.

\bibitem {Sh} Z.W. Shen,
On fundamental solutions of generalized Schr\"odinger operators.
 {\it J. Funct. Anal.}, {\bf 167} (1999), no. 2, 521--564.



\bibitem{SW} E. M. Stein and G. Weiss,
\textit{Introduction to Fourier Analysis on Euclidean spaces},
Princeton Univ. Press,
Princeton, NJ, 1970.

\bibitem{Ta} M.E. Taylor, {\it Microlocal analysis on Morrey spaces. Singularities and oscillations}
(Minneapolis, MN, 1994/1995), 97¨C135, IMA Vol. Math. Appl., 91, Springer, New York, 1997.



\end{thebibliography}
\end{document}